\documentclass[11pt]{amsart}
\usepackage{amssymb, amsmath}
\usepackage{amsthm}
\usepackage{amscd}
\usepackage{graphicx}
\newtheorem{theorem}{Theorem}
\newtheorem{lemma}[theorem]{Lemma}

\newtheorem{corollary}[theorem]{Corollary}

\theoremstyle{definition}

\newtheorem*{remark}{Remark}

\title[Division in group rings of surface groups]
{Division in group rings of surface groups}
\author{Grigori Avramidi}
\address{Max Planck Institute for Mathematics\\
Bonn\\
Germany, 53111}

\def\ra{\rightarrow}

\def\beqa{\begin{eqnarray}}
\def\eeqa{\end{eqnarray}}
\def\beqa{\begin{eqnarray}}
\def\eeqa{\end{eqnarray}}

\DeclareMathOperator{\SO}{SO}

\usepackage{pict2e}

\input xy
\xyoption{all}
\begin{document}
\maketitle
\begin{abstract}
We prove a division algorithm for group rings of high genus surface groups and use it to show that some $2$-complexes with surface fundamental groups are standard. We also give an application of division to cohomological dimension of $2$-relator groups acting on $\mathbb H^n$.
\end{abstract}
\section{introduction}


The goal of this paper is to study $2$-complexes $X$ with a fixed fundamental group $\Gamma$ up to homotopy equivalence by means of a division algorithm over the group ring of $\Gamma$. These two things are related through the second homotopy group, which is a $\mathbb Z\Gamma$-module. Most of the mathematical content of the paper consists of proving a division algorithm for group rings of high genus surface groups. We find this interesting in its own right, even outside the context of $2$-complexes.  

\subsection*{On division} In the rational group ring of a free group there is a division algorithm analogous to polynomial long divison that was discovered by Moritz Cohn \cite{cohn}. A division algorithm is a process that lets one divide one element $x$ by another non-zero element $y$ with a remainder $r$ whose `size' is smaller than that of $y$. In the group ring $\mathbb QF_n$, the measure of `size' we use is the diameter of the support of the group ring element (defined at the end of this section), which we denote by $|\cdot|$. In symbols, a division algorithm asks for $q,r\in\mathbb QF_n$ such that $x=qy+r$ and $|r|<|y|$ or $r=0$. Unlike in the case of polynomial long division, there cannot be a division algorithm for nonabelian free groups that works for arbitrary $x$ and $y$. In fact, for a generic pair of group ring elements, the diameter of the support of any linear combination will be at least as large as that of either element, so there is no hope of obtaining a remainder of smaller diameter.\footnote{To take a concrete example, in the free group $F_2$ on the letters $g$ and $h$, any linear combination of $g-1$ and $h-1$ has diameter $\geq 1$, so it is not possible to divide $g-1$ by $h-1$ with remainder of zero diameter.} Therefore, in order to have hope there must be linear combinations of $x$ and $y$ of small diameter. What Cohn discovered is that there is a division algorithm {\it as long as $x$ and $y$ satisfy a non-trivial linear relation in the group ring}.\footnote{In the commutative case $\mathbb QF_1=\mathbb Q[t,t^{-1}]$ any pair of elements $x$ and $y$ satisfy the obvious relation $xy-yx=0$, and the algorithm becomes the usual long division for Laurent polynomials.} This condition means there are elements $a,b\in\mathbb QF_n$, not both zero, such that $ax+by=0$. In fact, a geometric picture of this relation is what dictates the process for actually running the algorithm (see section \ref{freegroupalg}).


In this paper, we show that the same division algorithm is true when $\Gamma$ is the fundamental group of a surface of sufficiently high genus. 

\begin{theorem}[Division algorithm for surface groups]
\label{surfacealg}
Let $\Gamma$ be the fundamental group of a closed surface of genus $\geq e^{1000000}$. Suppose $x$ and $y\not=0$ are elements in $\mathbb Q\Gamma$ satisfying a nontrivial relation $ax+by=0$. Then there are $q,r\in\mathbb Q\Gamma$ such that 
$x=qy+r$ and $|r|<|y|$ or $r=0$. 
\end{theorem}

Our method is inspired by Hog-Angeloni's geometric proof of Cohn's division algorithm \cite{hogangeloni} and by Delzant's proof that groups rings of hyperbolic groups with large infimum displacement have no zero divisors \cite{delzant}. 

\subsection*{Euclid's algorithm for finding the greatest common divisor} The process of applying the division algorithm repeatedly to a pair of elements, dividing at each stage the divisor from the previous stage by the remainder is called Euclid's algorithm. Starting from the division algorithm in the integers (or in the polynomial ring $\mathbb Q[t]$) Euclid's algorithm produces the greatest common divisor of two integers (or polynomials). The same is true in our case.
    
\begin{corollary}[Euclid's algorithm for surface groups]
\label{euclid}
Applying the division algorithm repeatedly, first dividing $x$ by $y$ to obtain a remainder $r_1$, then dividing $y$ by $r_1$ to obtain a remainder $r_2$, and so on, eventually produces an element $z:=r_k$ that divides
the previous $r_{k-1}$ with no remainder. The element $z$ obtained in this way is a greatest\footnote{We say $z$ is a {\it divisor} of $x$ if $x=az$ for some $a\in\mathbb Q\Gamma$. It is a {\it greatest common divisor} of $x$ and $y$ if $z$ is a divisor of $x$ and $y$ and for any other divisor $z'$ of $x$ and $y$, $z'$ divides $z$. We say `a' here instead of `the' because greatest common divisors are only well-defined up to multiplication by a unit in $\mathbb Q\Gamma$.} common divisor of $x$ and $y$. 
\end{corollary} 

\subsection*{Algebraic application}

Rephrasing things a bit, Euclid's algorithm implies that the (left) 
ideal $(x,y)$ generated by the pair of elements $x,y\in\mathbb Q\Gamma$ is always free: If $x$ and $y$ do not satisfy any relation, then they are a free basis for the ideal, and if they do satisfy a relation then the ideal is generated by their greatest common divisor $z$. But, by the theorem of Delzant alluded to earlier, $z$ is not a zero-divisor, which is the same as saying that the ideal $z$ generates is free. A similar argument shows any pair of vectors $v,w\in\mathbb Q\Gamma^n$ generate a free $\mathbb Q\Gamma$-module, and all the arguments work over any field, in particular over the finite fields $\mathbb F_p$.

\begin{corollary}
\label{freefield}
For any field $k$, any submodule $M$ of $k\Gamma^n$ generated by a pair of vectors is free.
\end{corollary} 
For topological applications, we need this sort of result over $\mathbb Z\Gamma$. 
Using a `local-to-global' method of Bass (\cite{bass}) we assemble the $\mathbb Q$ and $\mathbb F_p$ statements together to prove such a result {\it under the additional assumption\footnote{An assumption {\it is} needed: the ideal $(2,t-1)$ in $\mathbb Z[\mathbb Z]=\mathbb Z[t,t^{-1}]$ is not free even though all ideals in $k[\mathbb Z]$ are.} that the quotient $\mathbb Z\Gamma^n/M$ is torsion-free}. This is good enough for us since the topologically meaningful 
modules associated to a $2$-complex satisfy this condition.

\begin{corollary}
\label{freez}
If a submodule $M$ of $\mathbb Z\Gamma^n$ is generated by a pair of vectors and $\mathbb Z\Gamma^n/M$ is torsion-free, then $M$ is free. 
\end{corollary} 

\subsection*{Non-free examples}

To put the division algorithm and its corollaries into context, note that corollary \ref{freez} is false for the group $\mathbb Z^2$: The ideal $(s-1,t-1)$ in $\mathbb Z[\mathbb Z^2]= \mathbb Z[s,s^{-1},t,t^{-1}]$ is not free since it has the obvious relation $(s-1)(t-1)=(t-1)(s-1)$ and cannot be generated by one element. More generally, for any non-free group $\Gamma$ generated by a pair of elements $a$ and $b$, the ideal $(a-1,b-1)$ in $\mathbb Z\Gamma$ is not free.\footnote{If $(a-1,b-1)$ is free then $\Gamma$ has cohomological dimension one, hence is free by Stallings' theorem (\cite{stallings}).} Remarkably, there is a $2$-generator, $2$-relator group that, by Thurston's work (4.7 of \cite{thurstonnotes}), arises as the fundamental group of a closed hyperbolic $3$-manifold obtained by Dehn filling on the figure-eight knot complement (see section \ref{2relator}). So, the division algorithm and its corollaries do not extend to fundamental groups of arbitrary hyperbolic manifolds.  

\subsection*{Group theoretic application}
Our proof of the division algorithm does work word-for-word for any group that acts by isometries on hyperbolic space $\mathbb H^n$ with {\it large infimum displacement} (this as a quantitative improvement on torsion-freeness). Free groups and high genus surface groups are low-dimensional groups that have such actions. As an application of division, we show that any (cohomologically) higher dimensional group that has such an action requires more than two relations to present. In other words 

\begin{corollary}
\label{2rel}
Suppose $\Gamma$ is a finitely generated $2$-relator group acting by isometries on $\mathbb H^n$ with infimum displacement $\geq 2000$. Then $\Gamma$ has cohomological dimension $\leq 2$. $($\footnote{Torsion free $1$-relator groups have the stronger property of having aspherical presentation complexes (\cite{cockcroft}).}$)$
\end{corollary}
It seems clear that the method should work for $\delta$-hyperbolic groups of large infimum displacement and carrying out the details of this might make a good Master thesis. 

\subsection*{Topological application}
Let us now turn to the topological application mentioned at the beginning of the introduction. An old theorem of Tietze \cite{fox} says that two $2$-complexes with the same fundamental group and Euler characteristic become homotopy equivalent after wedging both of them with the same suffiiently large number of $2$-spheres. A basic question is to determine whether wedging on these extra $2$-spheres is really necessary. One of the first examples of inequivalent $2$-complexes with the same fundamental group and Euler characteristic involves the trefoil group $T=\left<a,b\mid a^2=b^3\right>$. Let $Y$ be the presentation $2$-complex corresponding to this standard presentation. Dunwoody constructed another presentation $2$-complex $X$ for the trefoil group whose second homotopy group $\pi_2X$ is not free as a $\mathbb ZT$-module (\cite{dunwoody}). This complex has two generators and two relations so it has the same Euler characteristic as $Y\vee S^2$, but is not homotopy equivalent to it ($\pi_2(Y\vee S^2)$ is free since $Y$ is aspherical). Dunwoody also showed that the complexes $X$ and $Y\vee S^2$ do become homotopy equivalent after wedging on another $S^2$, which on the level of $\pi_2$ says that $\pi_2 X\oplus\mathbb ZT=\mathbb ZT\oplus\mathbb ZT$. So, $\pi_2X$ is generated by two elements and is stably free but not free.\footnote{The Klein bottle group $K=\left<a,b\mid a^2b^2=1\right>$ also has such stably free but not free $\mathbb ZK$-modules generated by a pair of elements, but they have not yet been geometrically realized as $\pi_2$-modules of $2$-complexes (\cite{harlander}).} Corollary \ref{freez} implies that this algebraic phenomenon does not happen for fundamental groups $\Gamma$ of high genus surfaces.

We can also ask whether a similar topological phenomenon to the one discovered by Dunwoody can happen for surface groups $\Gamma=\pi_1\Sigma$ in place of the trefoil group $T$. If $X$ is a $2$-complex with surface fundamental group and minimal Euler characteristic $\chi(X)=\chi(\Sigma)$, then it is easy to see that $X$ is homotopy equivalent to $\Sigma$. The first interesting case when the Euler characteristic is non-minimal is $\chi(X) =\chi(\Sigma)+1$. The main point is to show $\pi_2X$ is free. One way\footnote{Another way to show $\pi_2X$ is free, which also works for the groups in Corollary \ref{2rel}, is given in section \ref{2relator}.} is to use a theorem of Louder (\cite{louder}) which implies (see section \ref{topology}) that $X$ becomes standard 
after wedging on $\#(2$-cells of $X)-(\chi(X)-\chi(\Sigma))$ different $2$-spheres. So, if $X$ has two $2$-cells then $X\vee S^2$ is homotopy equivalent to $\Sigma\vee S^2\vee S^2$. On $\pi_2$, this implies $\pi_2X$ is stably free and generated by two elements. If the surface has high enough genus then Corollary \ref{freez} implies that $\pi_2X$ is free, and hence $X$ is homotopy equivalent to $\Sigma\vee S^2$. In summary 

\begin{theorem}
\label{standard}
Suppose $X$ is a $2$-complex with two $2$-cells and surface fundamental group $\pi_1 X=\pi_1\Sigma$. If the genus of the surface is $\geq e^{1000000}$, then $X$ is homotopy equivalent to $\Sigma$ or $\Sigma\vee S^2$. 
\end{theorem}

\subsection*{On $2$-complexes with more $2$-cells}
Let us finish this introduction with several remarks about generalizations to $2$-complexes with more than two $2$-cells. 

For the torus group $\mathbb Z^2$ not every submodule of a free $\mathbb Z[\mathbb Z^2]$-module is free, but all the stably-free ones are (this is Serre's conjecture proved by Quillen and Suslin, see \cite{lam}), and this is all one needs to show that any $2$-complex with $\mathbb Z^2$ fundamental group is standard. 
For the free groups $F_m$, there is a generalization of Euclid's algorithm (also due to Cohn) which shows that any ideal in $\mathbb QF_m$ (on any finite number of generators) is free. It also works with coefficients in $\mathbb F_p$ instead of $\mathbb Q$ so Bass's theorem implies any stably free $\mathbb ZF_n$-module is free. This implies all finite $2$-complexes with free fundamental group are standard. (See \cite{hogangeloni}.) 

On the other hand, the fundamental group of an orientable genus $g$ surface does have a non-free ideal on $2g$ generators, namely its augmentation ideal. It is tempting to conjecture that any ideal on fewer than $2g$ generators is free. Let us only remark here that Cohn's generalized Euclid's algorithm also has a surface version, which shows that any ideal on $f(g)$ generators is free, where $f(g)$ is a function such that $f(g)\ra\infty$ as $g\ra\infty$. The details of this are more involved than the $2$-generator case and will (?) be the subject of a future paper. Combining such an algorithm with Bass's method and Louder's result would imply that $2$-complexes with $\Sigma_g$ fundamental group and $\leq f(g)$ $2$-cells are standard.   

\subsection*{Plan of the paper} We explain the division algorithm for free groups in Section \ref{freegroupalg}. In Section \ref{deltahyp} we recall and derive properties of hyperbolic space that will be used in the proof of Theorem \ref{surfacealg} (the division algorithm for surface groups), which is given in Section \ref{surfacealgproof}. We then give a proof of Euclid's algorithm for surface groups together with Corollaries \ref{freefield} and \ref{freez} in Section \ref{algebra}. The group theoretic application (Corollary \ref{2rel}) and one way to get Theorem \ref{standard} is proved in Section \ref{2relator} and the other way is given in Section \ref{topology}.

\subsection*{Notation and terminology} Before we start, let us fix some notation that will be used throughout the paper and describe how group ring elements can, to a large extent, be thought of geometrically. 

Throughout the paper $\Gamma$ will denote either a free group or a surface group. The group $\Gamma$ acts by a covering space action on a space $Y$, which is a tree when $\Gamma$ is free and the hyperbolic plane $\mathbb{H}^2$ when $\Gamma$ is a surface group. Pick an orbit of $\Gamma$ in $Y$ and identify group elements with points of that orbit in $Y$. A group ring element $x\in\mathbb Q\Gamma$ is a finite formal linear combination $x=\sum x_{\gamma}\cdot \gamma$.

\subsubsection*{Support}
The support of $x$ consists of all the group elements $\gamma$ with non-zero coefficients $x_{\gamma}$ appearing in this sum, thought of as points in $Y$. We will denote the support of an element by the corresponding capital letter. So, the support of $x$ will be denoted $X$.

\subsubsection*{Diameter}
The diameter of $X$ is the maximal distance between a pair of points in $X$. It will be denoted $|x|$ (or $|X|$), and we will also call it the diameter of $x$. 

\subsubsection*{Barycenter}
The barycenter of $X$ is the center of the smallest ball containing $X$. It will be denoted $\widehat x$ (or $\widehat X$), and we will simply call it the barycenter of $x$. 

\subsubsection*{Boundary points}
Let $B_{\widehat x}(R)$ be the smallest ball containing $X$. We will call points of $X$ that are a maximal distance $R$ from the barycenter the boundary points of $x$.


\subsection*{Acknowledgements} I would like to thank T$\hat{\mbox{a}}$m Nguy$\tilde{\hat{\mbox{e}}}$n Phan for suggesting writing section \ref{freegroupalg}, Ian Leary and Jean Pierre Mutanguha for pointing out some $3$-dimensional $2$-relator groups, and the Max Planck Institute for its hospitality and financial support. 





\section{\label{freegroupalg}The division algorithm for free groups}
In this section, we will sketch the division algorithm for free groups. 

\subsection*{The algorithm}
Suppose we have a pair of group ring elements $x$ and $y$ that are related by a nontrivial linear relation $ax+by=0$. The main step in the division algorithm is to show that if $|x|\geq|y|$ then we can subtract translates of $y$ from $x$ to obtain an element $x_1=x-c_1y$ whose diameter is strictly smaller than that of $x$. Iterating this step will give division (we will say a few more words about this iteration at the end of this subsection.) 

The choice of $c_1$ is dictated by the relation $ax+by=0$ as follows. Let $o$ be the barycenter of the support of $ax$ and $R$ the radius of the smallest ball containing this support. 
There is an $x$-translate $\gamma x$ with $a_{\gamma}\not=0$ that contains a boundary point of $ax$. We can assume that $\gamma=1$, so that $x$ contains a boundary point. (If $\gamma\not=1$, multiply the relation on the left with $\gamma^{-1}$ and start again.) Let us call the points of $x$ that are on the boundary of $ax$ the {\it extremal points of} $x$. 
\begin{figure}[h!]
\centering
\includegraphics[scale=0.13]{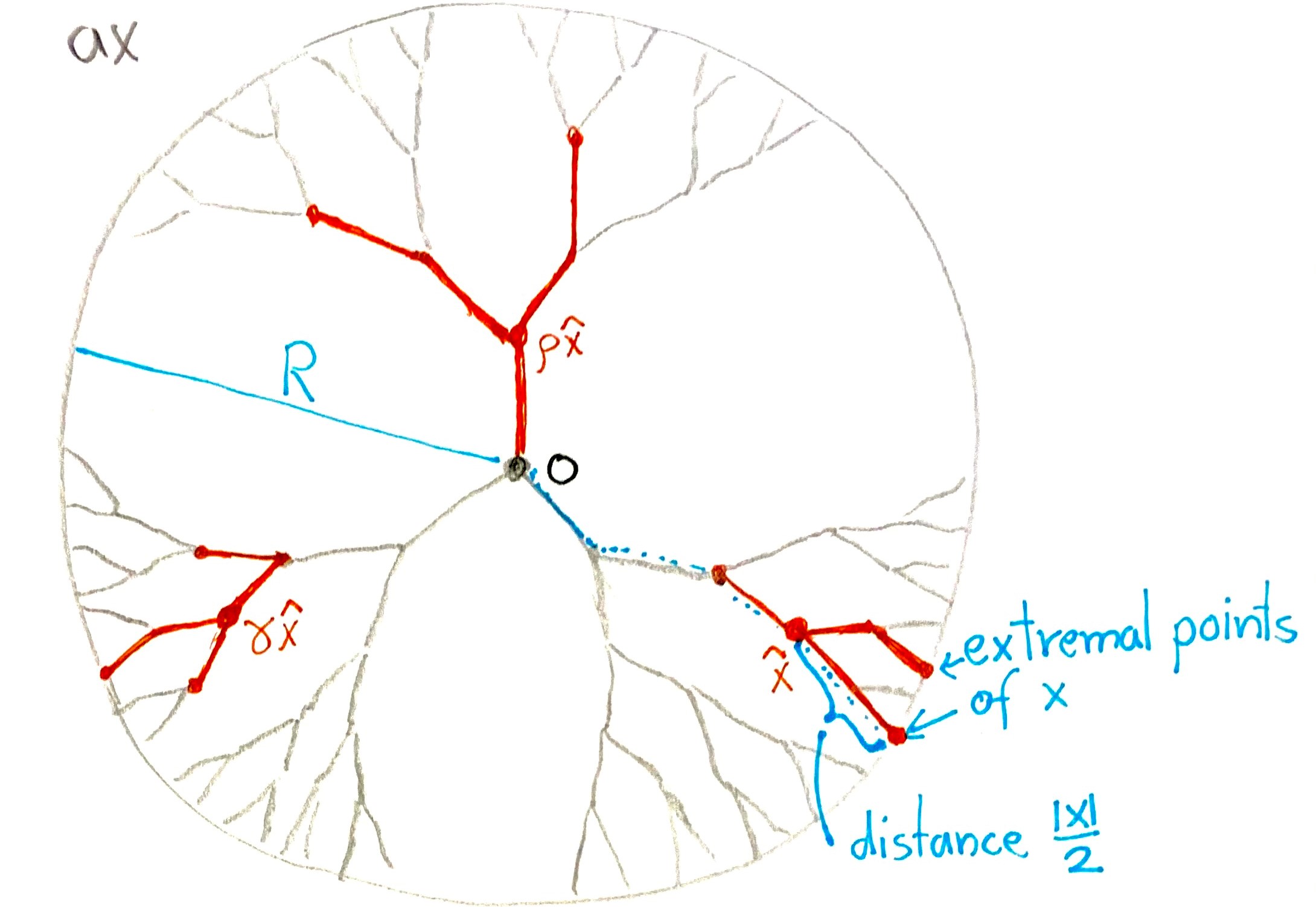}
\end{figure}
\newline

\noindent
{\bf Claim 1:} Any boundary point of $ax=-by$ appears in a unique $x$-translate (that is, $\gamma x$ with $a_{\gamma}\not=0$) and also in a unique $y$-translate ($\rho y$ with $b_{\rho}\not=0$.)
\newline

Therefore, the extremal points of $x$ can all be canceled by $y$-translates (weighted with appropriate coefficients)\footnote{The coefficients are $c_{\gamma}=-b_{\gamma}/a_1$ if $\gamma y$ contains an extremal point of $x$, and $c_{\gamma}=0$ otherwise.} to obtain an element 
$$
x_1=x-\sum c_{\gamma}\gamma y
$$ 
whose support does not contain any of the extremal points from the support of $x$. 
\newline

\noindent
{\bf Claim 2:} $x_1$ has smaller diameter than $x$. 
\newline

If $|x_1|<|y|$ then this finishes the division algorithm, since we can take $x_1$ to be the remainder. If not, then we note that $x_1$ and $y$ are related by the non-trivial relation $ax_1+(b+ac_1)y=0$ and repeat the above argument. Each iteration decreases the diameter by at least one, so after finitely many steps we arrive at an element $x_n=x_{n-1}-c_ny=x-(c_1+\dots+c_n)y$ whose diameter is smaller than $y$. This is our remainder.

\subsection*{Why it works} The key behind everything is that we are on a tree. 

Denote the support of $x$ by the corresponding capital letter $X$. We look at the set 
$$
S=\bigcup_{a_{\gamma}\not=0}\gamma X.
$$ 
It contains the support of $ax$ but can be strictly bigger if $ax$ has some cancellation. Let $B_{o'}(R')$ be the smallest ball containing $S$. We will show that any point in $S\cap S_{o'}(R')$ is in the support of $ax$. For this, it is enough to show that any $p\in S\cap S_{o'}(R')$ lies in precisely one $X$-translate. 
\begin{proof}
If $\gamma X$ touches the boundary at $p$ then, since we are on a tree, the barycenter $\gamma\widehat x$ lies on the geodesic from $o'$ to $p$ and is precisely $|X|/2$ away from $p$. If there is another translate $\rho X$ containing $p$ then $\gamma\widehat x=\rho\widehat x$ and hence $\gamma=\rho$. So, the translates $\rho X$ and $\gamma X$ are the same. 
\end{proof}
\begin{figure}[h!]
\centering
\includegraphics[scale=0.13]{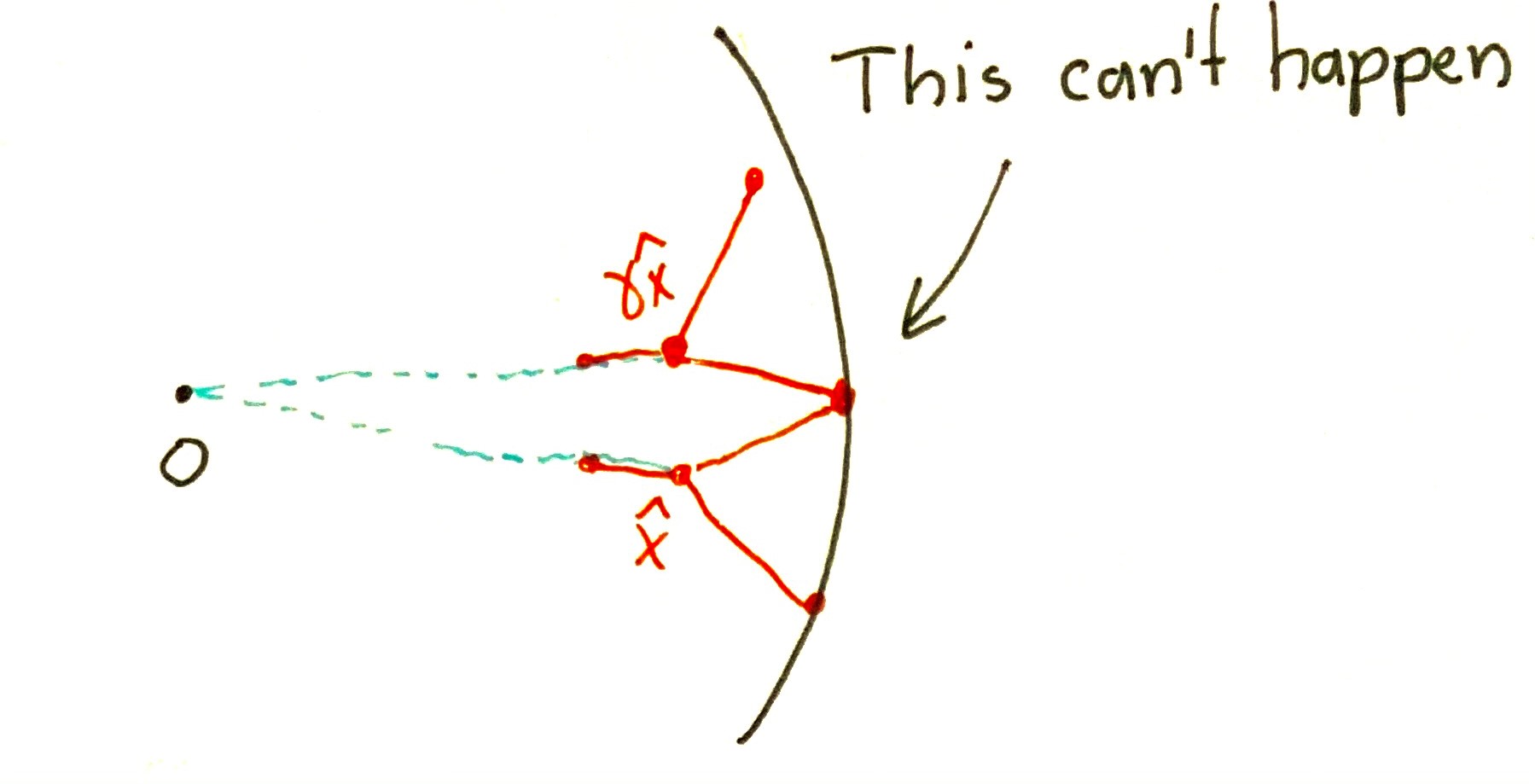}
\end{figure}

It follows from this that the points $S\cap S_{o'}(R')$ all appear in the support of $ax$. Therefore $o=o', R=R'$, what we have called above the `boundary points of $ax$' are pIn other words, therecisely the set $S\cap S_{o}(R)$, and every boundary point of $ax$ appears in exactly one $x$-translate. All the same arguments apply to the expression $by$. This proves the first claim.
\begin{remark}
It also shows that the picture of $ax$ on the previous page is accurate: The minimal ball containing the support of $ax$ entirely contains the supports of all the $x$-translates $\{\gamma x\}_{a_{\gamma}\not=0}$. 
\end{remark}
 
To prove the second claim, one uses similar arguments to show (see figures on the next page) that all the points of $x_1$ are $\leq|x|/2$ away from the barycenter $\widehat x$ and are not extremal.  Thus, $|x_1|\leq|x|$. In the case of equality there is a diameter realizing segment in $x_1$ whose midpoint is $\widehat x$. But then, at least one of its endpoints is extremal, which is a contradiction. 

\begin{figure}[h!]
\centering
\includegraphics[scale=0.13]{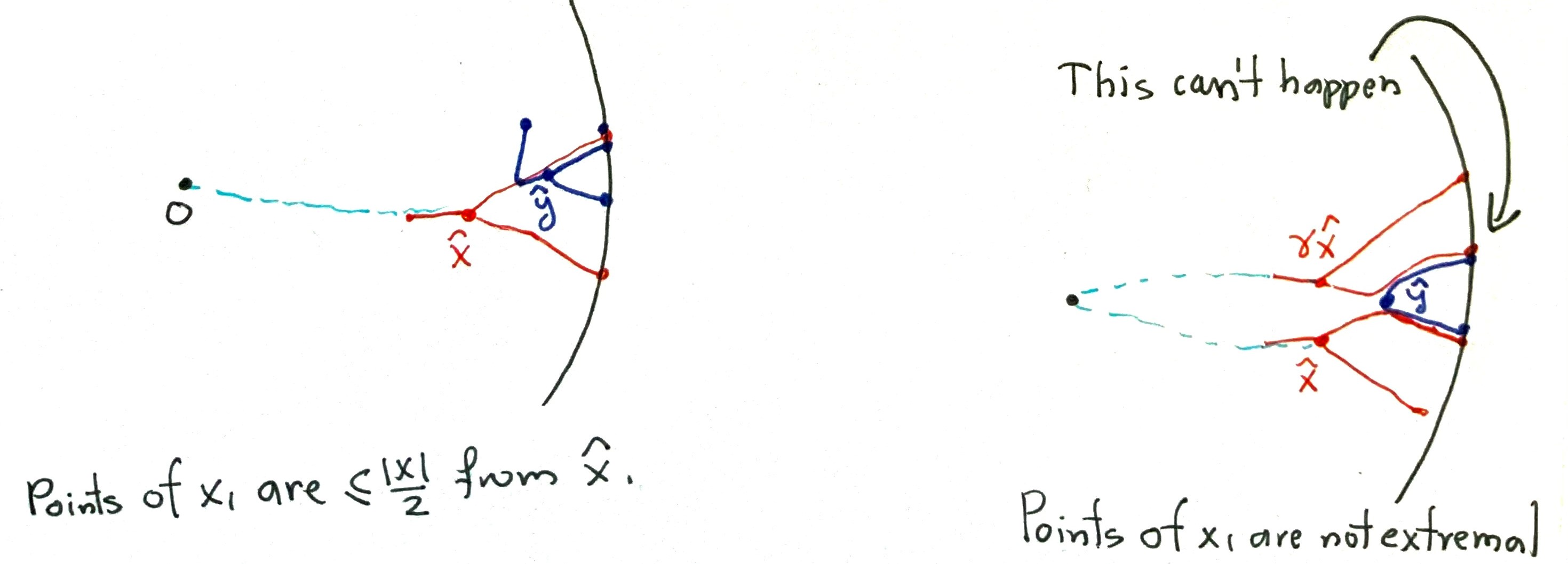}
\end{figure}

\subsection*{Where do relations come from?}
We can work backwards, starting from an element $z$ to produce pairs of elements satisfying successively more complicated relations: $(z,0)\ra(z,az)\ra(z+baz,az)\ra(z+baz,az+cz+cbaz)\ra\dots$. What the division algorithm implies is that any pair s of theatisfying a relation is obtained by this process. 


\section{\label{deltahyp}Tree-like properties of hyperbolic space}
Our proof of the division algorithm for surface groups is based on the tree-like properties of hyperbolic space. In this section we recall these properties in a convenient form and derive some specific consequences that will be used in the proof.

\subsection{$\delta$-hyperbolicity}
Everything can be easily obtained from the following basic property.
\begin{itemize}
\item
There is a universal constant $\delta$ so that if $pq$ is a segment with midpoint $m$ and $o$ is any point in hyperbolic space, then one of the paths $omp$ or $omq$ cannot be shortened by more than $\delta$.
In symbols
\begin{equation*}
\label{hyp}
\max(d(o,p),d(o,q))\geq d(o,m)+{1\over 2}d(p,q)-\delta.
\end{equation*}
\end{itemize}
\begin{remark}
In a tree we can take $\delta=0$ and in hyperbolic space we can take $\delta=5$. 
\end{remark}
It is useful to note that one of the angles $\angle_m(o,p)$ or $\angle_m(o,q)$ is obtuse ($\geq\pi/2$), and the maximum is achieved for the endpoint corresponding to this obtuse angle. It follows that 
\begin{itemize}
\item
any geodesic segment connecting a sphere $S_o(R)$ to a larger concentric sphere $S_o(R')$ and not intersecting the interior of $B_o(R)$ has length between $|R'-R|$ and $|R'-R|+\delta$.
\end{itemize} 
\begin{figure}[h!]
\centering
\includegraphics[scale=0.16]{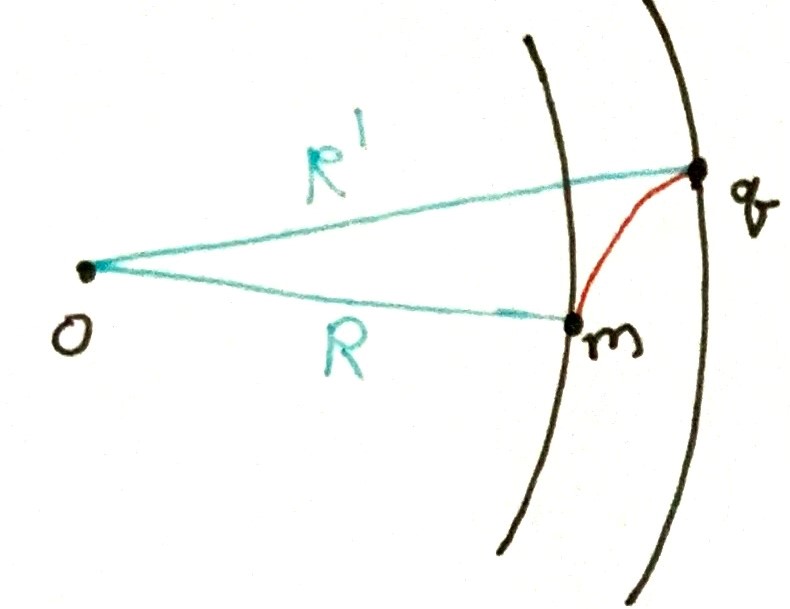}
\end{figure}

\begin{proof}
Let $m$ be a point on $S_o(R)$ and $q$ a poin of thet on $S_o(R')$. The angle $\angle_m(o,q)$ is obtuse, so $d(o,q)\geq d(o,m)+d(m,q)-\delta$. Plugging in $d(o,q)=R'$ and $d(o,m)=R$ gives $d(m,q)\leq R'-R+\delta$. The other inequality $R'-R\leq d(m,q)$ is clear. 
\end{proof}

\subsection{Midpoints and barycenters}
A consequence of $\delta$-hyperbolicity is that if $pq$ is a length $L$ segment in an $R$-ball, then its midpoint $m$ is within $R-L/2+\delta$ of the center of the ball. 

\begin{proof}
Let $o$ be the center of the $R$-ball. Then $R\geq\max(d(o,p),d(o,q))\geq d(o,m)+L/2-\delta$. 
\end{proof}
Another consequence is that any set $X$ of diameter $D$ is contained in a $(D/2+\delta)$-ball.
\begin{proof}
Let $p,q$ be a pair of points realizing the diameter $D$ and let $m$ be their midpoint. If $o$ is any point in $X$ then $D\geq \max(d(o,p),d(o,q))\geq d(o,m)+D/2-\delta$ implies $d(o,m)\leq D/2+\delta$. In other words, $X$ is contained in the $D/2+\delta$ ball centered at $o$.  
\end{proof}
These two properties together imply that 
\begin{itemize}
\item
the barycenter of a set is $2\delta$-close to the midpoint of any segment realizing the diameter.
\end{itemize}
So we can replace one with the other at the expense of a small error.
\newline

Next, suppose that $X$ is a set of diameter $D$, $\widehat x$ is its barycenter and $o$ is a point. 
Then for any diameter realizing segment $pq$ of $X$ with midpoint $m$ we have 
\begin{eqnarray}
\max(d(o,p),d(o,q))&\geq& d(o,m)+{D\over 2}-\delta\\
\label{top}
&\geq& d(o,\widehat x)+{D\over 2}-3\delta.
\end{eqnarray}
For any point $p'$ in $X$ we have $d(o,p')\leq d(o,\widehat x)+d(\widehat x,p')\leq d(o,\widehat x)+D/2+\delta$ and therefore
\begin{equation}
\label{bottom}
d(o,\widehat x)\geq d(o,p')-{D\over 2}-\delta, 
\end{equation}
 
Putting these two inequalities together tells us how far the barycenter $\widehat x$ is from a point $o$ in terms of the diameter of $X$ and the radius of the smallest ball at $o$ containing $X$.
\begin{lemma}
If $B_o(R)$ is the smallest ball centered at $o$ containing $X$, then 
$$
R-{D\over 2}-\delta\leq d(o,\widehat x)\leq R-{D\over 2}+3\delta.
$$
\end{lemma}
\begin{proof}
Plug $d(o,p')=R$ into (\ref{bottom}) and $\max (d(o,p),d(o,q))\leq R$ into (\ref{top}). 
\end{proof}

\subsection*{Shrinking the diameter of $X$}
Another application of these two inequalities specifies particular points of $X$ to throw out in order to shrink its diameter.
\begin{lemma}[Extremal cancellation]
\label{5delta}
If $B_o(R)$ is the smallest ball centered at $o$ containing $X$, then the diameter of $X\cap B_o(R-5\delta)$ is strictly less than the diameter of $X$.
\end{lemma}  
\begin{proof}
If the diameter of $X\cap B_o(R-5\delta)$ is not smaller, one of its diameter realizing segments $pq$ also realizes the diameter of $X$. Therefore, plugging (\ref{bottom}) into (\ref{top}) and using $R=d(o,p')$ gives $\max(d(o,p),d(o,q))\geq R-4\delta$, so at least one of the points $p$ or $q$ is outside the $R-5\delta$ ball centered at $o$, which is a contradiction. 
\end{proof}

\newpage
\subsection{Fellow traveling}
The next treelike feature of hyperbolic space we need is {\it fellow traveling}. It says that for a pair of points $p$ and $q$ on the boundary of a ball centered at $o$, the segments $pq$ and $po$ fellow travel until we reach the midpoint of $pq$, up to an error $4\delta$. To express it precisely, it is useful to parametrise geodesics. For a geodesic segment $pq$ we denote by $pq(t)$ the point obtained by traveling from $p$ to $q$ for a time $t$ along the geodesic.

\begin{lemma}[Fellow traveling property]
For a pair of points $p,q\in S_o(R)$ and $t\leq{d(p,q)\over 2}$ we have
$$
d(pq(t),po(t))\leq 4\delta.
$$
\end{lemma}
\begin{figure}[h!]
\centering
\includegraphics[scale=0.20]{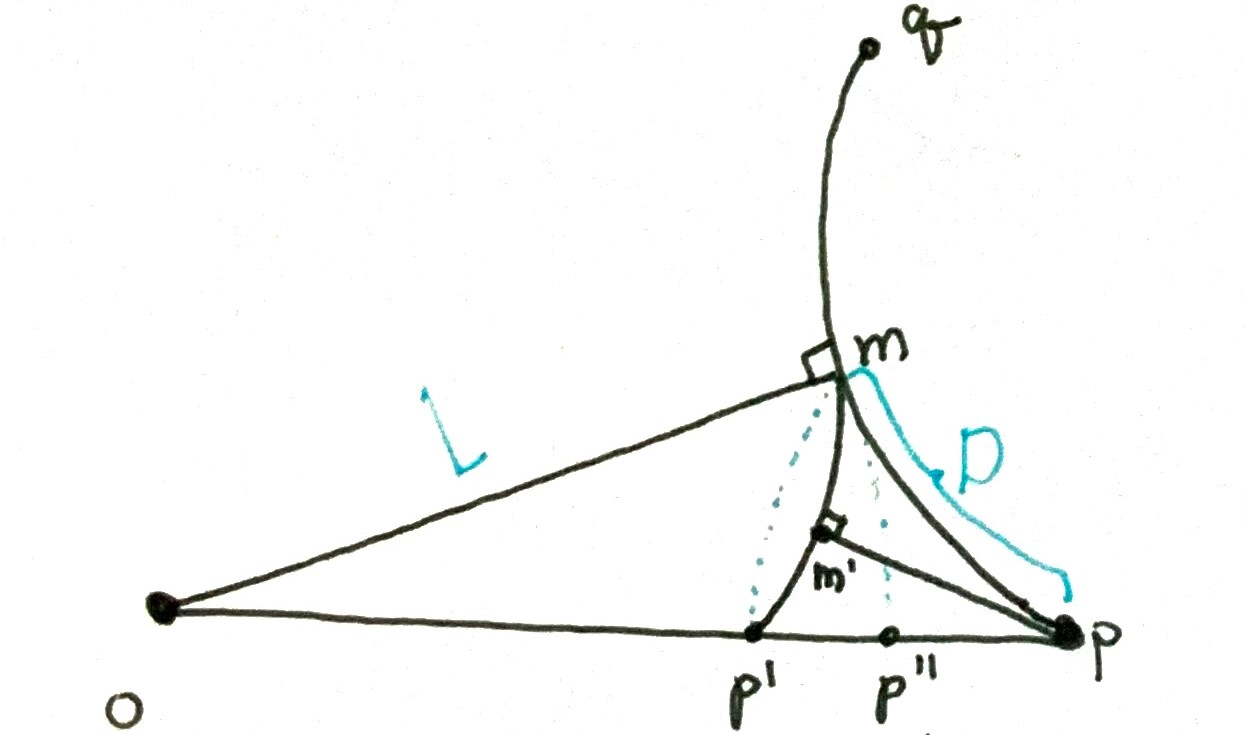}
\end{figure}

\begin{proof}
Let $m$ be the midpoint of $pq$, $L=d(o,m)$ and $D=d(m,p)$. Let $p'=po(D)$ be the point obtained by going for a time $D$ from $p$ to $o$ and $p''=op(L)$ the point obtained by traveling for a time $L$ from $o$ to $p$. Finally, let $m'$ be the midpoint of the geodesic segment $mp'$. Now, since the angle $\angle_m(o,p)$ is right, it follows that 
$$
R\geq L+D-\delta.
$$ 
It is also clear from the picture that 
$$
d(p,m')\geq d(p'',p)=R-L,
$$ 
and plugging in the previous inequality gives
$$
d(p,m')\geq D-\delta. 
$$
Since the angle $\angle_{m'}(p,p')$ is right, if follows that 
$$
D\geq d(p,m')+d(m',p')-\delta.
$$ 
Therefore 
$$
d(m',p')\leq D+\delta-d(p,m')\leq 2\delta.
$$
Since $m'$ is the midpoint of $mp'$, it follows that $d(m,p')\leq 4\delta$. This proves the lemma for $t=d(p,q)/2$. The lemma for smaller values of $t$ follows from convexity. 
\end{proof}

\subsection*{Large infimum displacement implies no zero divisors in the group ring}
Next, we give a key application of fellow traveling. It\footnote{To be more precise, the $\mu=0$ case in the setting of $\delta$-hypebolic groups.} was observed by Delzant in \cite{delzant} and is the main step in his proof that group rings of some hyperbolic groups have no zero divisors.

\begin{lemma}[Delzant]
\label{delzant}
Suppose $\gamma$ is an isometry of $\mathbb H^n$. If $X$ and $\gamma X$ are contained in a ball $B_o(R)$ and their intersection contains a point $p$ in the $\mu$-neighborhood of the boundary of the ball, then the midpoint $m$ of the segment from $p$ to $\gamma p$ is moved $\leq\mu+9\delta$ by $\gamma^{-1}$.
\end{lemma} 

\begin{figure}[h!]
\centering
\includegraphics[scale=0.13]{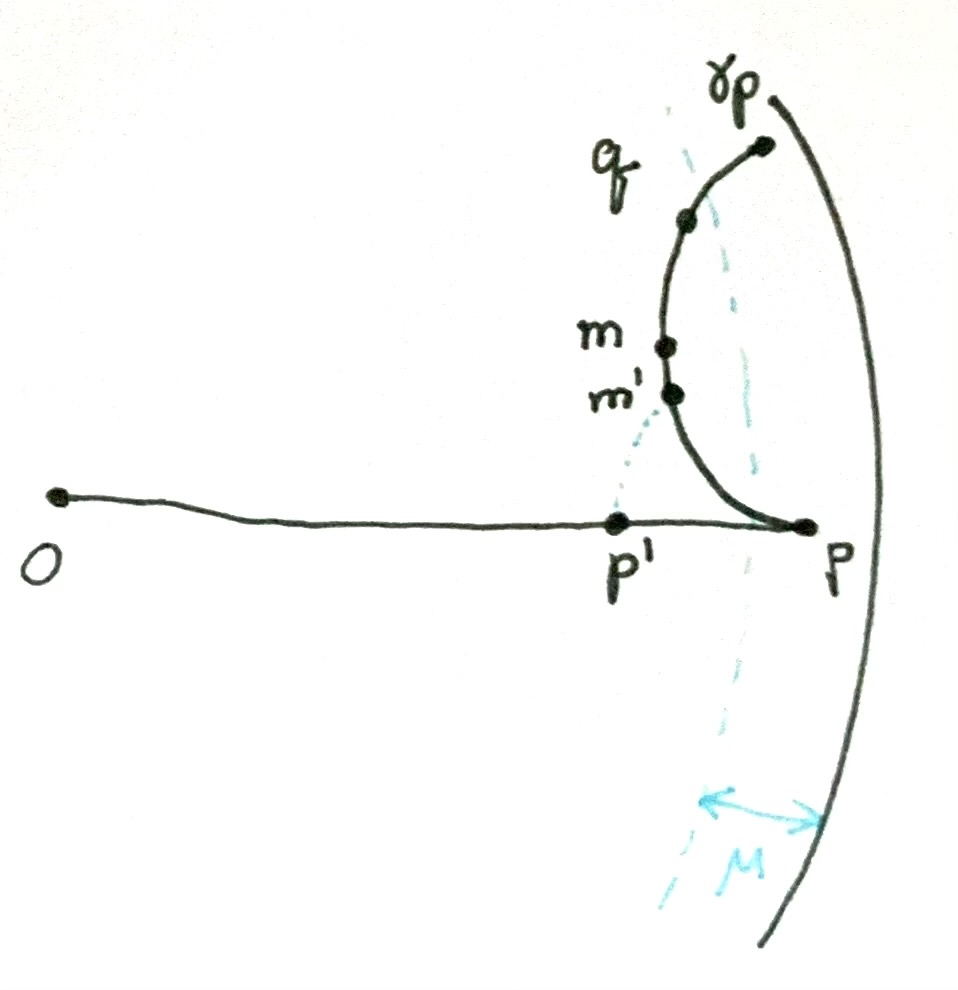}
\end{figure}

\begin{proof}
Let $L$ be the length of the segment from $p$ to $\gamma p$. Let $q$ be the point obtained by going from $\gamma p$ to $p$ for a distance $\mu+\delta$. Then $q\in B_o(R-\mu)$. Let $m'$ be the midpoint of the segment from $p$ to $q$. Then $po$ fellow travels with $pq$ for a time $t={L-(\mu+\delta)\over 2}$ until it reaches $m'$ so, if we denote by $p'=po(t)$ the point reached by traveling from $p$ to $o$ for a time $t$, then 
$$
d(m,p')\leq d(m,m')+d(m',p')\leq {\mu+\delta\over 2}+4\delta.
$$
The same argument applied to the segment from $p$ to $\gamma^{-1}p$ shows that its midpoint $\gamma^{-1}m$ satisfies $d(\gamma^{-1}m,p')\leq {\mu+\delta\over 2}+4\delta$. Therefore $d(m,\gamma^{-1}m)\leq \mu+9\delta$. 
\end{proof}

In other words, if the infimum displacement of $\Gamma$ acting on $\mathbb H^n$ is $>\mu+9\delta$, then $\Gamma$-translates of $X$ that lie in a ball do not intersect in the $\mu$-neighborhood of the boundary of that ball. 

Now we apply this to products in the group ring. The following corollary will be used repeatedly in the next section. It implies that any cancellation in a product $ax$ happens away from the boundary of $ax$, as long as the infimum displacement is sufficiently large. 
\begin{corollary}
\label{products}
Suppose $\Gamma$ has infimum displacement $>\mu+9\delta$. Let $a$ and $x$ be non-zero group ring elements.  Then, the smallest ball containing $ax$ also contains all the $x$-translates $\left\{\gamma x\right\}_{a_{\gamma}\not=0}$. Moreover, every point in the $\mu$-neighborhood of the boundary of this ball is contained in at most one such $x$-translate. 
\end{corollary}

\begin{proof}
The support of the product $ax$ is contained in the set 
$$
S=\bigcup_{a_{\gamma}\not=0}\gamma X.
$$
Let $B_o(R)$ be the smallest ball containing $S$. Delzant's lemma implies that on the $\mu$-neighborhood of the boundary of this ball the $X$ translates $\{\gamma X\}_{a_{\gamma}\not=0}$ do not intersect. Therefore, $B_o(R)$ is the smallest ball containing the support of $ax$. The rest is clear.
\end{proof}

In particular, this says that once the infimum displacement is $>9\delta$ the support of $ax$ is non-empty, so $ax\not=0$. 
\begin{corollary}[Delzant]
If $\Gamma$ has infimum displacement $>9\delta$ then $\mathbb Q\Gamma$ has no zero divisors. 
\end{corollary}


\subsection{Approximating barycenters}
The following lemma is useful. 
\begin{lemma}
\label{bary}
Suppose $X$ is a set with diameter $D$ and barycenter $\widehat x$ contained in a ball $B_o(R)$ and $q\in X$ is a point in the $\mu$-neighborhood of the boundary of the ball. Let $qo(D/2)$ be the point obtained by traveling for time $D/2$ along the geodesic from $q$ to $o$. Then
$$
d(qo(D/2),\widehat x)\leq 9\delta+{3\over 2}\mu.
$$
\end{lemma}

\begin{figure}[h!]
\centering
\includegraphics[scale=0.16]{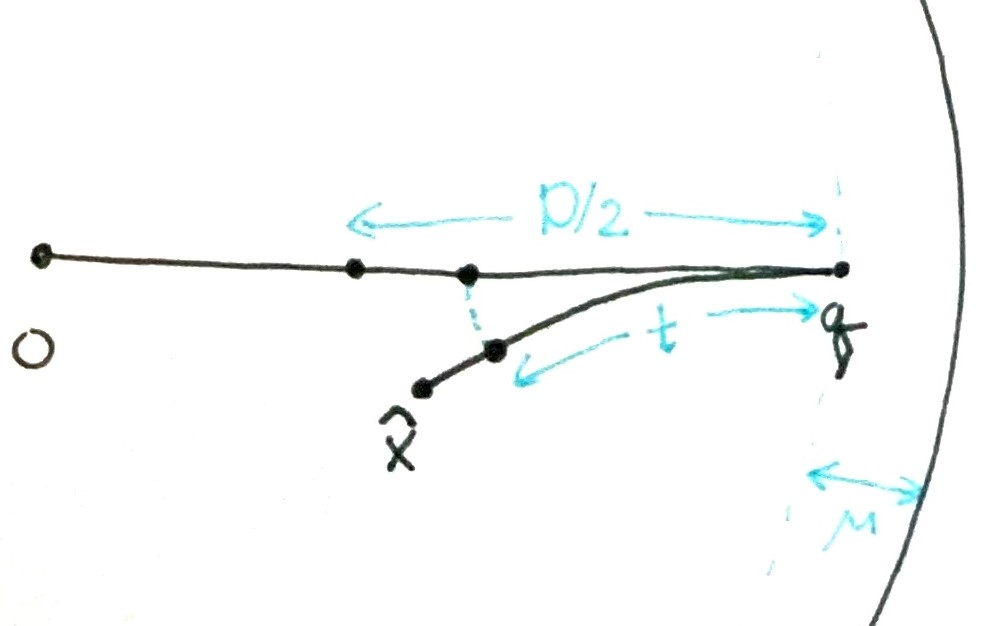}
\end{figure}

\begin{proof}
We can assume that $d(q,o)=R-\mu$. 
First, note that $d(\widehat x,o)\leq R-{D\over 2}+3\delta$ implies that the distance from $\widehat x$ to $S_o(R-\mu)$ is at least $s={D\over 2}-3\delta-\mu$. Therefore, the segment $q\widehat x$ can be extended by $s$ before it reaches $S_o(R-\mu)$, and hence the segment $q\widehat x$ fellow travels with $qo$ for a time $t={d(q,\widehat x)+s\over 2}$. Also note that the distance from $\widehat x$ to $q$ is controlled by 
$$
s\leq d(\widehat x,q)\leq {D\over 2}+\delta.
$$
Therefore 
$$
|d(q,\widehat x)-t|={d(q,\widehat x)-s\over 2}\leq {4\delta+\mu\over 2}
$$
while
$$
s\leq t\leq {D-2\delta-\mu\over 2}
$$ 
implies that 
$$
|t-D/2|\leq 3\delta+\mu.
$$
Thus, the distance from $\widehat x$ to $qo(D/2)$ is bounded by
\begin{eqnarray*}
d(\widehat x,qo(D/2))&\leq&|d(\widehat x,q)-t|+4\delta+|t-D/2|\\
&\leq&9\delta+{3\over 2}\mu. 
\end{eqnarray*} 
\end{proof}
An immediate consequence is the following.

\begin{corollary} 
\label{barydistance}
If $X$ and $Y$ are both sets in $B_o(R)$ containing a point $q$ in the $\mu$-neighborhood of the boundary, then the barycenters of $X$ and $Y$ are ${||X|-|Y||\over 2}$-apart, up to an error $18\delta+3\mu$.
\end{corollary} 

A more significant consequence for us is the following. 
\begin{corollary}
\label{jump}
Suppose $X,\gamma X,$ and $Y$ are contained in a ball $B_o(R)$, the intersection of $Y$ and $X$ contains $q$, and the intersection of $Y$ and $\gamma X$ contains $q'$, where $q$ and $q'$ are in the $\mu$-neighborhood of the boundary of the ball. If $|X|\geq|Y|$ then 
$$
d(\widehat x,\gamma\widehat x)\leq 36\delta+6\mu.
$$
\end{corollary}
\begin{figure}[h!]
\centering
\includegraphics[scale=0.15]{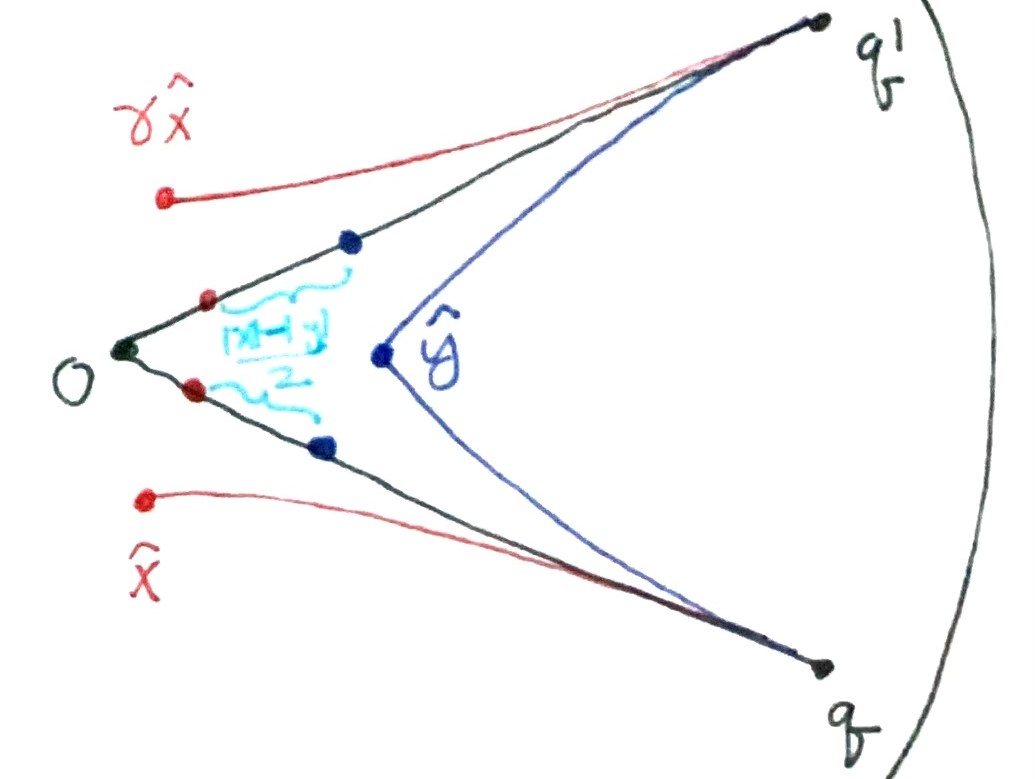}
\end{figure}

\begin{proof}
Look at the function $f(t)=d(qo(t),q'o(t))$. It is $\leq 18\delta+3\mu$ at $t=|Y|/2$ by Lemma \ref{bary}. It decreases as $t$ goes from $|Y|/2$ to $|X|/2$ by convexity (This is where we use $|X|\geq|Y|$). Finally, at $t=|X|/2$ it differs from $d(\widehat x,\gamma\widehat x)$ by an error of at most $18\delta+3\mu$, again by Lemma \ref{bary}. This establishes the corollary. 
\end{proof}
We will use these two corollaries in the proof of the division algorithm for surface groups in the next section. 

\section{\label{surfacealgproof}Proof of the division algorithm for surface groups}
\subsection{Setup}
Throughout the proof, we will keep track of how large the infimum displacement has to be for the argument to work at that stage. The infimum displacement needed to make everything work is stated at the end, where we also translate it into a condition about the genus of the surface. To start, we assume the infimum displacement is $>9\delta$ so that there are no zero divisors.

We are given a non-trivial relation $ax+by=0$ where $a, b, x$ and $y\not=0$ are elements of the group ring $\mathbb Q\Gamma$, and we want to show that there are $q,r\in\mathbb Q\Gamma$ such that $x=qy+r$ and $|r|<|y|$ or $r=0$. 
If $|x|<|y|$ then there is nothing to do, since we can take $q=0,r=x$. If $|x|\geq |y|$, then it is enough to subtract a multiple $b'y$ of $y$ from $x$ for which the resulting element $x'=x-b'y$ has smaller diameter than $x$. Since the set of possible diameters is discrete and the elements $x'$ and $y$ satisfy the non-trivial\footnote{If this relation were trivial then $a=0$ and the original relation would imply that $y$ is a zero-divisor.} relation $a(x-b'y)+(b+ab')y=0$, iterating the process finitely many times will prove the division algorithm.

Next, we will describe which points of $x$ we will try to cancel out with translates of $y$ in order to reduce the diameter of the resulting group ring element $x'$. This will be dictated by the relation $ax+by=0$. Let $o$ be the barycenter of the support of $ax$ and $R$ the radius of $ax$. As long as the infimum displacement is $>9\delta$, by Corollary \ref{products}, all the $x$-translates $\{\gamma x\}_{ a_{\gamma}\not=0}$ are contained in the ball $B_o(R)$. Pick such an $x$-translate $\gamma x$ containing a boundary point of $ax$. After multiplying our relation on the left by $(a_{\gamma}\gamma)^{-1}$, we can assume that this translate is $x$, i.e. that $x$ contains a boundary point of $ax$ and that $a_1=1$. 

Let us call the points of $x$ that are in the $\alpha$-neighborhood of the boundary of $ax$ the {\it $\alpha$-extremal} points of $x$. We have shown in Lemma \ref{5delta} that if we throw out the $5\delta$-extremal points from the support of $x$, then the resulting set has strictly smaller diameter. So these are the points we will try cancel out. To that end, note that if the infimum displacement is $>5\delta+9\delta$ then, by Corollary \ref{products}, all these $5\delta$-extremal points are not contained in any other $x$-translate $\gamma x$ with $a_{\gamma}\not=0$.  The relation $ax=-by$ implies that each one of them must be contained in a $y$-translate $\gamma y$ with $b_{\gamma}\not=0$, and Corollary \ref{products} applied to $by$ implies that there is a {\it unique} such $y$ translate. 

Therefore, as long as the infimum displacement of $\Gamma$ is sufficiently large ($>14\delta$), the $5\delta$-extremal points of $x$ can all be cancelled by $y$-translates (weighted with appropriate coefficients)\footnote{To be more precise, $c_{\gamma}=-b_{\gamma}$ if $\gamma y$ contains a $5\delta$-extremal point of $x$ and $c_{\gamma}=0$ otherwise.}. Call the resulting element 
$$
x'=x-\sum c_{\gamma}\gamma y.
$$ 
Our goal in the rest of the proof is to show the diameter of $x'$ is less than the diameter of $x$. 
\subsection{Showing $|x'|<|x|$}
We will warm up by showing that if the diameter of $x'$ is greater than that of $x$, it cannot be much greater. We can estimate the distance from a $y$-point $p$ of $x'$ to $\widehat x$ using the estimate on the distance between barycenters given in Corollary \ref{barydistance}:
\begin{eqnarray*}
d(\widehat x,p)&\leq& d(\widehat x,\gamma\widehat y)+d(\gamma\widehat y,p)\\
&\leq&\left({|x|-|y|\over 2}+18\delta+3\cdot 5\delta\right)+\left({|y|\over 2}+\delta\right)\\
&=&{|x|\over 2}+34\delta.
\end{eqnarray*}
So, if $|x'|\geq|x|$ it follows that a diameter realizing segment of $x'$ has midpoint $35\delta$-close to $\widehat x$. Note for future use that this implies one of the endpoints of this segment is $37\delta$-extremal.

Let us look more at the extra $y$-points $p$ of $x'$ that have been introduced by subtracting the $y$-translates $\{\gamma y\}_{c_{\gamma}\not=0}$ from $x$. Fix a constant $\mu\geq 37\delta.$ First, we will show that if the infimum displacement is sufficiently large ($>36\delta+6\mu$), then such a $y$-point $p$ cannot be $\mu$-extremal. If it was, then it would have to cancel with a unique $x$-translate $\rho x$ that is different from $x$.\footnote{If $p$ canceled with $x$ then it would not have appeared in $x'$.} But then the barycenters $\widehat x$ and $\rho\widehat x$ would be too close! More precisely, we would have $d(\widehat x,\rho\widehat x)\leq 36\delta+6\mu$ by Corollary \ref{jump}, which contradicts the infimum displacement assumption. Therefore, if $|x'|\geq|x|$ then $x'$ has a $37\delta$-extremal $x$-point. Call this point $q'$.  

\begin{remark}
In the case of free groups acting on trees, $\delta=0$ and above we can take $\mu=0$ so that at this stage in the argument we have an element $x'$ with $|x'|\leq|x|$ and if $|x'|=|x|$ then $x'$ has an extremal point, which is not a $y$-point, hence must be an $x$ point. But we assumed that all the extremal $x$-points have been canceled out, so we arrive at a contradiction. In the surface group case we have to work harder. The reason is because we have found a $37\delta$-extremal $x$-point $q'$ in $x'$, while only the $5\delta$-extremal $x$-points have been canceled out.
\end{remark}

Now, for large enough infimum displacement $(>37\delta+9\delta)$ the point $q'$ appears in a unique $y$ translate $\rho y$ that is different from all the $y$-translates $\{\gamma y\}_{c_{\gamma}\not=0}$ that we subtracted from $x$ to get $x'$. By Corollary \ref{barydistance} we have 
$$
d(\widehat x,\rho\widehat y)\leq{|x|-|y|\over 2}+18\delta+3\cdot 37\delta.
$$ 

The rest of the argument breaks up into two cases, depending on the size of $|x|-|y|$. 

\subsection*{First, we deal with the case is $|x|-|y|\leq\mu$} 
In this case, the barycenter of $\rho y$ and of all the $y$-translates $\{\gamma y\}_{c_{y}\not=0}$ are $\left({\mu\over 2}+129\delta\right)$-close to $\widehat x$. Therefore, if the infimum displacement is large enough $(>\mu+258\delta)$ we must have $\rho\widehat y=\gamma\widehat y$ and hence $\rho=\gamma$, which is a contradiction. 


\subsection*{Finally we deal with the case $|x|-|y|\geq\mu$} 
We will show that in this case the $y$-points of $x'$ are $<|x|/2-\delta$ away from $\widehat x$. This will imply that the diameter of $x'$ is less than $|x|$ and we will be done.

Let $p$ be a $y$-point of $x'$. Thus, there is a $y$-translate $\gamma y$ and a $5\delta$-extremal point $q$ of $x$ so that both $p$ and $q$ are in $\gamma y$. Denote by $L$ the length of the segment $pq$ and $m$ its midpoint. Since $p$ is not $\mu$-extremal, the segment $qp$ can be extended by $\mu-5\delta$ before it reaches the $5\delta$-neighborhood of the boundary. Let $m'$ be the midpoint of this extended segment. It fellow travels with the segment $qo$ for a distance $t={L+\mu-5\delta \over 2}$. Let $y_0=qo(t)$ be the point obtained by traveling from $q$ to $o$ for a time $t$. Also, let $x_0=qo(|x|/2)$ be the point obtained by traveling from $q$ to $o$ for a time $|x|/2$. This is illustrated in the figure below.

\begin{figure}[h!]
\centering
\includegraphics[scale=0.16]{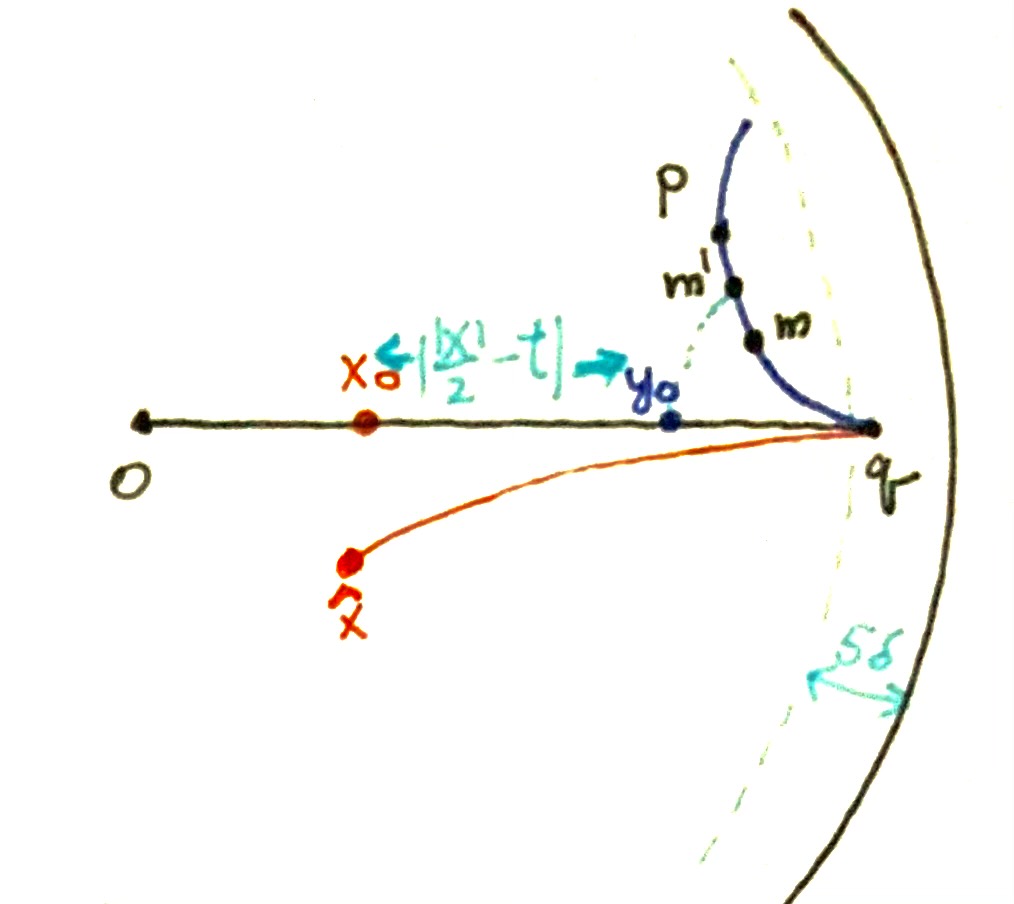}
\end{figure}

Note that ${|x|\over 2}-t={|x|-L-(\mu-5\delta)\over 2}$ is non-negative because of the case we are in, so we can compute
\begin{eqnarray*}
d(\widehat x,p)&\leq& d(\widehat x, x_0)+d(x_0,y_0)+d(y_0,m')+d(m',p)\\
&\leq& \left(9\delta+{3\over 2}\cdot 5\delta\right)+\left|{|x|\over 2}-t\right|+4\delta+{L-(\mu-5\delta)\over 2}\\
&=&{|x|\over 2}+25.5\delta-\mu.
\end{eqnarray*}
So, for this part of the argument to work, any $\mu>26.5\delta$ will do.

In summary, everything works for $\mu=37\delta$, and in that case the biggest displacement condition we need is that the infimum displacement is $\mu+258\delta=295\delta$. Since we are in hyperbolic space, we can take $\delta=5$. 

\subsection*{From infimum displacement to genus}  Buser showed in \cite{buser} that every surface of genus $\geq 2$ has a hyperbolic metric with infimum displacement (=length of shortest geodesic) $\geq 2\sqrt{\log(g)}$. For this metric and $g\geq e^{1000,000}$ we get infimum displacement $\geq 2000>295\cdot 5$, which is good enough. This finishes the proof. 
\begin{remark}
Buser and Sarnak show in \cite{busersarnak} that there is a sequence of hyperbolic surfaces $\Sigma_{g_i}$ with $g_i\ra\infty$ for which one has a much better bound, namely infimum displacement $\geq{4\over 3}\log g_i$ and that every genus $g$ surface has a hyperbolic metric with infimum displacement $c\log g$ where $c$ is some small (unspecified) positive constant that doesn't depend on the genus, but neither of these can be directly applied to get an explicit bound on how high the genus $g$ has to be.  
\end{remark} 
\begin{remark}
This last step is the only place in the proof where we use the fact that $\Gamma$ is a surface group. Everything else works word-for-word (with the same constants) for groups $\Gamma$ acting isometrically on hyperbolic $n$-space $\mathbb H^n$ with infimum displacement $\geq 295\cdot 5$. 
\end{remark}
\subsection*{A remark about fields}
Everything in the paper so far works with coefficients $\mathbb Q$ replaced by any field $k$, in particular by the finite fields $\mathbb F_p$. This will be used in the next section.

\section{\label{algebra}Euclid's algorithm and algebraic applications}

\subsection{Proof of Euclid's algorithm}
We are given a pair of elements $x,y\in\mathbb Q\Gamma$ satisfying a non-trivial relation $ax+by=0$. Dividing $x$ by $y$ we get $q_0$ and $r_0$ such that $x=q_0y+r_0$ and $|r_0|<|y|$ or $r_0=0$. If $r_0\not=0$ then the elements $y$ and $r_0$ satisfy the non-trivial relation $ar_0+(b+aq_0)y=0$. So we can divide $y$ by $r_0$ to get $q_1$ and $r_1$ such that $y=q_1r_0+r_1$ and $|r_1|<|r_0|$ or $r_1=0$, and so on. We iterate this process. Since at each step the diameter of the remainder decreases, the process stops after finitely many steps with an $r_k$ that divides $r_{k-1}$ without remainder. All the pairs produced in this way generate the same ideal $(x,y)=(y,r_0)=(r_0,r_1)=\dots=(r_{k-1},r_k)=(r_k,0)$. 

We now show that the last remainder $z=r_k$ is a greatest common divisor of $x$ and $y$. The element $z$ is a divisor of $x$ and $y$ since $x,y\in(z)$. Suppose $z'$ is another divisor such that $x=cz',y=c'z'$. Since $z\in(x,y)$ we can express it as $\mathbb Q\Gamma$-linear combination $z=a'x+b'y=(a'c+b'c')z'$, so $z'$ is a divisor of $z$. Therefore, $z$ is a greatest common divisor of $x$ and $y$. 

\subsection{Modules generated by pairs of vectors $v,w$ in $\mathbb Q\Gamma^n$ and $\mathbb F_p\Gamma^n$}

Delzant's result that $\mathbb Q\Gamma$ has no zero-divisors implies that the submodule of $\mathbb Q\Gamma^n$ generated by a single non-zero vector $v=(v_1,\dots,v_n)$ is free.\footnote{If there are no zero divisors, then the map $\mathbb Q\Gamma\ra \mathbb Q\Gamma^n,a\mapsto av$ is an isomorphism onto its image, which is the module generated by $v$.} 
Our division algorithm implies the analogous result for a pair of vectors. The proof is very similar to that of Euclid's algorithm. 
Since we will need both the $\mathbb Q$ and $\mathbb F_p$ versions in the next subsection, we state it for a general field $k$. 
\newline

\noindent
{\bf Corollary \ref{freefield}.}
{\it Let $k$ be a field. Any submodule $M$ of $k\Gamma^n$ generated by a pair of non-zero vectors $v, w$ is free.} 

\begin{proof}
We may assume that $v_1\not=0$ and that $|v_1|\geq|w_1|$. If for any relation $av+bw=0$ both $a$ and $b$ are zero, then $M$ is free of rank two. So, suppose there is such a relation with either $a$ or $b$ non-zero. We will show that this implies $M$ is free of rank one.

There are two cases to consider, depending on whether or not $w_1$ is zero. 

\subsection*{Case 1: $w_1=0$} Looking at the first coordinate of the relation, we get $av_1=0$ and since $v_1\not=0$ we must have $a=0$. Thus $bw=0$. Since the relation was non-trivial, $b\not=0$ so we must have $w=0$. But then $M$ is generated by a single vector $v$, and hence it is free of rank one. 

\subsection*{Case 2: $w_1\not=0$} Then the relation $av_1+bw_1=0$ implies both $a$ and $b$ are non-zero. We use this relation to divide $v_1$ by $w_1$ and get $v'=v-qw$ satisfying $|v_1'|<|w_1|$ or $v_1'=0$. Then the vectors $v',w$ still generate $M$ and either $v'_1=0$ or the sum of diameters of their first entries $|v_1'|+|w_1|$ is strictly smaller than $|v_1|+|w_1|$. Moreover, $av'+(b-aq)w=0$ is again a non-trivial relation (with $a\not=0)$. 

At this point, we have arrived back at the situation of the two cases, with $v_1'$ in place of $w_1$. Moreover, if $v_1'\not=0$ then the sum of diameters of the first entries of generators $|v_1'|+|w_1|$ is strictly smaller than $|v_1|+|w_1|$. Therefore, after iterating this process finitely many times it will stop and we will arrive in the case 1 situation with $M$ a free module generated by a single vector. 
\end{proof}

\subsection{Bass's `local-to-global' method for $\mathbb Z\Gamma$-modules}

Obviously any zero divisor in  $\mathbb Z\Gamma$ is a zero-divisor in $\mathbb Q\Gamma$, so submodules of $\mathbb Z\Gamma^n$ generated by a single vector are free. The analogous statement for modules generated by a pair of vectors is not true. For example, the ideal $(2,t-1)$ in the group ring $\mathbb Z[\mathbb Z]=\mathbb Z[t,t^{-1}]$ is not free even though all ideals in $k[\mathbb Z]$ are. A general `local-to-global'\footnote{It is a `local-to-global' theorem because it assembles $\mathbb Q$ and $\mathbb F_p$ results to get a $\mathbb Z$ result.} theorem of Bass (\cite{bass}) shows that this sort of thing doesn't happen when the module splits off as a direct summand of $\mathbb Z\Gamma^n$ (in other words, if the module is {\it projective}). This is good enough for the proof of Theorem \ref{standard} given in section \ref{topology}. The proof given below is Bass's argument, specialized to our situation.    

\begin{corollary}
\label{bass}
If a submodule of $\mathbb Z\Gamma^n$ generated by a pair of non-zero vectors $v,w$ splits off as a direct summand, then it is free.
\end{corollary}

\begin{proof}
If $v$ and $w$ do not satisfy any non-trivial relation, then $(v,w)$ is a free $\mathbb Z\Gamma$-module. If they satisfy a non-trivial relation, then they generate the same $\mathbb Q\Gamma$ module as their greatest common divisor $z\in\mathbb Q\Gamma^n$. We rescale $z$ (multiplying by a rational number if necessary) so that $z\in(v,w)$ and $z\notin(kv,kw)$ for any integer $k>1$. Since $z$ is a $\mathbb Q\Gamma$-divisor of $v$ and $w$, there is a positive integer $m$ such that $mv=az,mw=bz$ for some $a,b\in\mathbb Z\Gamma$. Pick the smallest such $m$. In summary we have sandwiched the module generated by $z$ in the following way:
$$
(mv,mw)\subset (z)\subset(v,w).
$$ 
Our goal is to show that $m=1$. Suppose it is not, and let $p$ be a prime dividing $m$. 
For any $\mathbb Z\Gamma$-module $M$, denote the mod $p$ reduction by $M_p:=M/pM$. Note that the composition of induced maps 
\begin{equation}
\label{composition}
(mv,mw)_p\ra(z)_p\ra (v,w)_p
\end{equation} 
is zero because $p$ divides $m$. The key is to show the second map is injective. 
\subsection*{Claim: The map $i:(z)_p\ra(v,w)_p$ is injective}
This is where we use the assumption that $(v,w)$ is a direct summand of $\mathbb Z\Gamma^n$. It implies that the inclusion $(v,w)\hookrightarrow \mathbb Z\Gamma^n$ induces an inclusion of mod $p$ reductions $(v,w)_p\hookrightarrow\mathbb F_p\Gamma^n$. By Corollary \ref{freefield}, $(v,w)_p$ is a free $\mathbb F_p\Gamma$-module. So, the image of $i$ is a submodule of a free module and generated by one element, so it is free by Corollary \ref{freefield}. Therefore, $i$ is either injective or the zero map. The later happens precisely if $z\in (pv,pw)$, but this is ruled out by our choice of $z$. So, the map $i$ is injective. 

Since the composition (\ref{composition}) is zero, this implies the first map $(mv,mw)_p\ra(z)_p$ is the zero map, which is the same as saying $(mv,mw)\subset(pz)$. But then $z$ is a $\mathbb Z\Gamma$-divisor of both ${m\over p}v$ and ${m\over p}w$, which contradicts the minimality of $m$. So we are done.
\end{proof}

\subsection{Proof of Corollary \ref{freez}}
Note that all we used in the proof above is that the inclusion $M\hookrightarrow \mathbb Z\Gamma^n$ induces an inclusion on mod $p$ reductions. For this we don't really need $M$ to be a direct summand. Knowing that {\it the quotient $Q=\mathbb Z\Gamma^n/M$ is torsion-free as an abelian group} is good enough:
Suppose $v\in M$ and $v=pw$ for some $w\in\mathbb Z\Gamma^n$. In the quotient $Q$ we have $\overline v=0$ and since $Q$ is torsion-free also $\overline w=0$. But that means $w\in M$ and therefore $v\in pM$. So, we've shown that $M\cap p\mathbb Z\Gamma^n=pM$, which is the same as saying that $M_p\ra\mathbb F_p\Gamma^n$ is injective.  
So, the proof of Corollary \ref{bass} also applies to such $M$. This proves Corollary \ref{freez}.

\section{\label{2relator}$2$-relator groups acting on hyperbolic space}

Until now, we have primarily focused on surface groups in this paper. They are two-dimensional groups acting on hyperbolic space, and requiring a single relation to present. Moreover, passing to finite index subgroups we get surface groups again but now of higher genus and (if we pick the subgroup correctly) with large infimum displacement. In higher dimensions $n\geq 3$ we can start with an arithmetically constructed uniform lattice in $\SO(n,1)$ and then pass to a deep enough congruence subgroup to get a group action with large infimum displacement on the hyperbolic space $\mathbb H^n$ (\footnote{For example, the group of invertible $4\times 4$ matrices with integer entries, preserving the form $x_1^2+x_2^2+x_3^2-7x_4^2$ and congruent to the identity matrix modulo $N$ for large enough $N$ acts on $\mathbb H^3=\SO(3,1)/\SO(3)$ in this way.}). It well known that---in contrast to surface groups---these higher dimensional lattices require more than one relation. In fact, they require more than two. The reason is that these uniform lattices have a non-zero cohomology class in dimension $n\geq 3$, while we will show next that $2$-relator groups acting with sufficiently large displacement on $\mathbb H^n$ are cohomologically $2$-dimensional. 

Shifting attention for a moment to these $2$-relator groups $\Gamma$ and $2$-complexes $X$ presenting them as such, the freeness of $\pi_2X$ comes out in the wash. A new wrinkle is that we do not know whether such $2$-relator groups have aspherical presentation $2$-complexes $Y$. For any that do (in particular, for the high genus surface groups) it is clear what a {\it standard} $2$-complex with fundamental group $\Gamma$ is (one homotopy equivalent to $Y\vee S^2\vee\dots\vee S^2$) and we get a version of Theorem \ref{standard} from the introduction. In summary, we have 



\begin{corollary}
\label{2relcor}
Suppose $X$ is a finite $2$-complex with two $2$-cells and fundamental group $\Gamma$. If $\Gamma$ acts isometrically on $\mathbb H^n$ with infimum displacement $\geq 2000$, then
\begin{itemize}
\item
the cohomological dimension of $\Gamma$ is $\leq 2$,
\item
$\pi_2X$ is free, and 
\item
if $\Gamma$ has an aspherical presentation $2$-complex $Y$ then $X$ is standard.\footnote{Homotopy equivalent to $Y$ or $Y\vee S^2$ or $Y\vee S^2\vee S^2$. The third case happens only if $\Gamma$ is a free.}
\end{itemize} 
\end{corollary}
\begin{proof}
Look at the chain complex on the universal cover:
$$
\pi_2(X)\ra C_2(\widetilde X)\ra C_1(\widetilde X)\ra C_0(\widetilde X)\ra\mathbb Z. 
$$ 
The image of the second map is called the relation module $R$. It is generated by two elements, a submodule of a free module, and the quotient $C_1/R$ is again a submodule of a free module. Therefore, by the remark at the end of the previous section, we can apply Bass's method to conclude $R$ is a free $\mathbb Z\Gamma$-module. Since $R$ is also the kernel of the third map, we get a free resolution $R\ra C_1\ra C_0\ra \mathbb Z$ of length $2$. This is the same as saying that the cohomological dimension of $\Gamma$ is $\leq 2$, so we have proved the first bullet. 

Since $R$ is free, $C_2$ splits as a direct sum $\pi_2(X)\oplus R$. So, $\pi_2(X)$ is a stably free $\mathbb Z\Gamma$-module generated by two elements. Corollary \ref{freez} implies it is free. This proves the second bullet. 

Finally, suppose there is an aspherical presentation $2$-complex $Y$. Start by building an arbitrary $\pi_1$-isomorphism $Y\ra X$. Since $\pi_2X$ is free, we can extend it to a homotopy equivalence from a standard complex $Y$ or $Y\vee S^2$ or $Y\vee S^2\vee S^2$ by mapping the $2$-spheres to a basis for $\pi_2X$. In the third case $\pi_2X=\mathbb Z\Gamma^2$, so the relation module $R$ vanishes, so $\Gamma$ has cohomological dimension one and hence, by Stallings' theorem (\cite{stallings}), is a free group. 
\end{proof} 
But, if $\Gamma$ does not have an aspherical presentation $2$-complex, then it is conceivable that there is a pair of $2$-complexes $X$ and $X'$ that have the same $\pi_2$ but are not homotopy equivalent. 

\subsection*{Flat and hyperbolic $3$-dimensional $2$-relator groups}
Torsion-free $1$-relator groups have aspherical presentation $2$-complexes (\cite{cockcroft}), so they are at most $2$-dimensional. This is no longer true for $2$-relator groups. 

The simplest $3$-dimensional example of a $2$-relator group was pointed out to me by Ian Leary. It is the fundamental group of the mapping torus of the matrix $\left(\begin{array}{cc}0&1\\-1&0\end{array}\right)$ acting on $\mathbb T^2$. Note that this is a closed, flat\footnote{The manifold is flat since it is obtained by gluing the ends of $\mathbb T^2\times[0,1]$ by an isometry.} $3$-manifold, so the fundamental group is $3$-dimensional. It has a $3$-generator and $3$-relator presentation $\left<a,b\mid[a,b]=1, tat^{-1}=b,tbt^{-1}=a^{-1}\right>$. One can eliminate the generator $b$ to get a $2$-generator, $2$-relator presentation. 

There are also hyperbolic $3$-manifold examples that were explained to me by Jean Pierre Mutanguha. The mapping torus of the matrix $\left(\begin{array}{cc}2&1\\1&1\end{array}\right)$ acting on the punctured torus is a hyperbolic $3$-manifold with a single cusp\footnote{This manifold is homeomorphic to the figure-eight knot complement (see p. 177 of \cite{thurston}).}. Its presentation is $\left<a,b,t\mid tat^{-1}=a^2b,tbt^{-1}=ab\right>$ and since the second relation says $a=[t,b]$ one can elliminate $a$ together with this relation to get a 1-relator presentation. One can close off the cusp by gluing in a solid torus, and for all but finitely many choices of gluing parameters (a pair of relatively prime numbers $(p,q)$) one gets a closed hyperbolic $3$-manifold (see 4.7 in \cite{thurstonnotes}). On the level of fundamental groups, the gluing introduces a new relation of the form $t^p=[a,b]^q$. So, one ends up with a closed hyperbolic $3$-manifold whose fundamental group has a $2$-generator $2$-relator presentation 
$$
\left<b,t\mid t[t,b]t^{-1}=[t,b]^2b,\hspace{0.3cm}t^p=[[t,b],b]^q\right>.
$$

\section{\label{topology}An improved Tietze's theorem for surface fundamental groups}
An old theorem of Tietze \cite{fox} says that two $2$-complexes with the same fundamental group become homotopy equivalent after wedging both of them with enough $2$-spheres. 
This section is about improvements on this theorem when the fundamental group is that of a closed surface $\Sigma$. The main point is to interpret a Nielsen equivalence result of Louder in this light. 

\subsection*{Minimal Euler characteristic} First note that if $X$ is a $2$-complex with fundamental group $\pi_1\Sigma$ and minimal Euler characteristic $\chi(X)=\chi(\Sigma)$, then $X$ is homotopy equivalent to $\Sigma$. 
\begin{proof}
The complexes become homotopy equivalent after wedging both with the same large number of $2$-spheres $n$. Since $\Sigma$ is aspherical, on $\pi_2$ this homotopy equivalence gives $\pi_2 S\oplus\mathbb Z\Gamma^n\cong\mathbb Z\Gamma^n$. So (see e.g. \cite{montgomery}) $S$ is also aspherical, and hence homotopy equivalent to $\Sigma$.
\end{proof}
\subsection*{Nielsen equivalence for surface groups}
The orientable surfaces have presentations 
$$\left<x_1,y_1,\dots,x_g,y_g\mid [x_1,y_1]\cdots[x_g,y_g]=1\right>$$ while the nonorientable ones have presentations $$\left<x_1,\dots,x_r\mid x_1^2\cdots x_r^2=1\right>.$$ A {\it standard generating set} is one of these, possibly with some extra generators $z_1,\dots,z_k$ satisfying the trivial relations $z_1=1,\dots,z_k=1$ thrown in at the end. 

Now, let $X$ be a finite presentation $2$-complex with $n$ generators $e_1,\dots,e_n$ for the surface group, and fix a $\pi_1$-isomorphism $f:X\ra\Sigma$. In \cite{louder}, Louder showed 
\begin{itemize}
\item There is a free group automorphism $\varphi:F_n\ra F_n$ so that $f\circ\varphi(e_1),\dots,f\circ\varphi(e_n)$ is a standard generating set for $\pi_1\Sigma$.
\end{itemize} 
\subsection*{Interpretation as quantitative variant of Tietze's theorem for surface groups}
For concreteness, suppose it is one representing a genus $g$ orientable surface with $k$ trivial generators at the end (the argument in the non-orientable case is similar). Form a new complex 
$$
Y=X\cup D^2_0\cup D^2_1\cup\dots\cup D^2_k
$$
by attaching $k+1$ different $2$-cells to $X$. The disk $D^2_0$ is attached along the commutator $[\varphi(e_1),\varphi(e_2)]\cdots[\varphi(e_{2g-1}),\varphi(e_{2g})]$ and the other disks $D^2_i$ are attached along $\varphi(e_{2g+i})$. By construction, these attaching maps are nullhomotopic in $\pi_1X$, so $Y$ is homotopy equivalent $X\vee S^2\vee\dots\vee S^2$. On the other hand, the map $f$ extends to $Y$ and its restriction to the union $S=D_0^2\cup\dots\cup D_k^2$ is a $\pi_1$-isomorphism. Since $f:S\ra\Sigma$ is a $\pi_1$-isomorphism that extends to the $2$-cells of $X$, the attaching maps of the $2$-cells of $X$ are null-homotopic in $S$, and we conclude that $Y$ is also homotopy equivalent to $S^2\vee\dots\vee S^2\vee S$. Finally, since the $2$-complex $S$ has the minimal possible Euler characteristic $\chi(S)=\chi(\Sigma)$ among $2$-complexes with this fundamental group, the map $f:S\ra\Sigma$ is a homotopy equivalence. In summary, $X$ becomes standard after wedging on $k+1$ different $2$-spheres: 
$$
X\vee (k+1)S^2\sim\Sigma\vee (\# \mbox{ of } 2\mbox{-cells of } X)S^2. 
$$
\subsection*{Second proof of Theorem \ref{standard}} 
The situation that our division algorithm can say something about is when $X$ has one vertex, two $2$-cells and Euler characteristic $\chi(X)=\chi(\Sigma)+1$. In this case, it is easy to see\footnote{In general, $k+1=\#(2\mbox{-cells of }X)-(\chi(X)-\chi(\Sigma))$} that $k=0$ and the above homotopy equivalence becomes
$$
X\vee S^2\sim\Sigma\vee S^2\vee S^2. 
$$
On $\pi_2$ this says $\pi_2X\oplus \mathbb Z\Gamma\cong\mathbb Z\Gamma^2$. Therefore, by Corollary \ref{freez}, $\pi_2X$ is free. From here we can proceed as in the proof of the third bullet of Corollary \ref{2relcor} to prove Theorem \ref{standard}.

\bibliography{2complexes}

\providecommand{\bysame}{\leavevmode\hbox to3em{\hrulefill}\thinspace}
\providecommand{\MR}{\relax\ifhmode\unskip\space\fi MR }
\providecommand{\MRhref}[2]{%
  \href{http://www.ams.org/mathscinet-getitem?mr=#1}{#2}
}
\providecommand{\href}[2]{#2}
\begin{thebibliography}{10}

\bibitem{bass}
H.~Bass, \emph{Projective modules over free groups are free}, Journal of
  Algebra \textbf{1} (1964), no.~4, 367--373.

\bibitem{buser}
P.~Buser, \emph{Riemannsche fl{\"a}chen mit grosser kragenweite}, CMH
  \textbf{53} (1978), no.~1, 395--407.

\bibitem{busersarnak}
P.~Buser and P.~Sarnak, \emph{On the period matrix of a riemann surface of
  large genus (with an appendix by jh conway and nja sloane)}, Inventiones
  mathematicae \textbf{117} (1994), no.~1, 27--56.

\bibitem{cockcroft}
W.~H. Cockcroft, \emph{On two-dimensional aspherical complexes}, PLMS
  \textbf{3} (1954), no.~1, 375--384.

\bibitem{cohn}
P.~M. Cohn, \emph{Free ideal rings and localization in general rings}, vol.~3,
  Cambridge university press, 2006.

\bibitem{fox}
R.~H. Crowell and R.~H. Fox, \emph{Introduction to knot theory}, vol.~57,
  Springer Science \& Business Media, 2012.

\bibitem{delzant}
T.~Delzant, \emph{Sur l'anneau d'un groupe hyperbolique}, Comptes Rendus
  \textbf{324} (1997), no.~4, 381--384.

\bibitem{dunwoody}
M.~J. Dunwoody, \emph{The homotopy type of a two-dimensional complex}, BLMS
  \textbf{8} (1976), no.~3, 282--285.

\bibitem{harlander}
J.~Harlander and A.~Misseldine, \emph{On the $ k $-theory and homotopy theory
  of the klein bottle group}, Homology, Homotopy and Applications \textbf{13}
  (2011), no.~2, 63--72.

\bibitem{hogangeloni}
C.~Hog-Angeloni, \emph{A short topological proof of cohn's theorem}, Topology
  and Combinatorial Group Theory, Springer, 1990, pp.~90--95.

\bibitem{lam}
T.-Y. Lam, \emph{Serre's problem on projective modules}, Springer Science \&
  Business Media, 2010.

\bibitem{louder}
L.~Louder, \emph{Nielsen equivalence in closed surface groups}, arXiv preprint
  arXiv:1009.0454 (2010).

\bibitem{montgomery}
M.~S. Montgomery, \emph{Left and right inverses in group algebras}, BAMS
  \textbf{75} (1969), no.~3, 539--540.

\bibitem{stallings}
J.~Stallings, \emph{On torsion-free groups with infinitely many ends}, Ann. of
  Math. (1968), 312--334.

\bibitem{thurstonnotes}
W.~P. Thurston, \emph{The geometry and topology of 3-manifolds}, Lecture note
  (1978).

\bibitem{thurston}
\bysame, \emph{Three dimensional manifolds, kleinian groups and hyperbolic
  geometry}, BLMS \textbf{6} (1982), no.~3, 357--381.

\end{thebibliography}
\bibliographystyle{amsplain}

\end{document}